\documentclass[letterpaper, 10 pt, conference]{ieeeconf}  
\IEEEoverridecommandlockouts    
\usepackage[compatibility=false]{caption}
\usepackage{silence}
\usepackage{subfig}

	\usepackage{cite}
    \usepackage{hyperref}

	\usepackage{amssymb,amsfonts,mathptmx}
	 
	\usepackage{amsthm}
	\usepackage{graphicx}
\usepackage{float}
	\usepackage{cite}
	\usepackage{enumerate}
	\usepackage{epsfig}
\usepackage[compatibility=false,font=small,labelfont=bf]{caption}
	\usepackage{amsmath} 
	\usepackage{bm}
	\usepackage{epstopdf,array}
	\usepackage{algorithm}
	\usepackage{algpseudocode}
	\usepackage{multirow,multicol}
	\let\labelindent\relax
	\usepackage{subfig,braket,enumitem,adjustbox,booktabs}
	\usepackage{color}
	\usepackage{float}
	\newfloat{algorithm}{t}{lop}
	\usepackage{pgf,tikz}
	\usetikzlibrary{arrows}
    \usepackage{xcolor}

\newcommand{\real}{\mathbb{R}}

\newcommand{\T}{^{\top}}

\newcommand{\grad}{\nabla}
\newcommand{\tr}{\operatorname{tr}}

\newcommand{\Delxi}{\Delta \xi}
\newcommand\pd{\partial}

\newcommand{\kb}{\overline{k}}
\newcommand{\td}{\text{d}}
\newcommand{\Dc}{\mathcal{D}}

\newcommand{\longthmtitle}[1]{\mbox{}{\textit{(#1).}}}

\newcommand{\oprocendsymbol}{\hbox{$\bullet$}}
\newcommand{\oprocend}{\relax\ifmmode\else\unskip\hfill\fi\oprocendsymbol}

	\newtheorem{theorem}{Theorem}[section]
	\newtheorem{proposition}[theorem]{Proposition}
	\newtheorem{lemma}[theorem]{Lemma}
	\newtheorem{corollary}[theorem]{Corollary}
	\newtheorem{remark}[theorem]{Remark}

\usepackage{enumitem}

\usepackage{float}
\begin{document}
\title{Control Co-design of systems with parabolic PDE dynamics }
\author{Antika Yadav$^*$ and Prasad Vilas Chanekar
\thanks{$^{*}$Corresponding author.} 
\thanks{This research is supported by SERB Start-up Research Grant - SRG/2023/001636 and SR/PURSE/2023/172, DST-PURSE, Govt. of India.}
\thanks{AY and PVC are with the Department of Electronics and
		Communication Engineering, Indraprastha Institute of Information Technology, New Delhi, India 110020, {\tt \{antikai,\,prasad\}}@iiitd.ac.in}}

\maketitle
\begin{abstract}
In this paper, we study the control co-design (CCD) synthesis problem for a class of systems with parabolic partial differential equation (PDE) dynamics. We first derive a sufficient stability condition for the PDE. By spatially discretizing the PDE and using the sufficient stability condition, we propose a computationally tractable approximate CCD problem.  We solve the approximate CCD problem using a gradient-based method.  Finally, we justify our proposed approach through an example.  
\end{abstract}
\section{Introduction}
 Systems whose dynamics is governed by partial differential equations (PDEs) are relevant in various physical and engineering applications such as thermal regulation, structural vibration suppression, chemical process control, etc. A significant challenge in designing PDE systems is their infinite-dimensional nature. 
 
 
\textit{Literature Review:} Boundary state feedback control is a very effective method in the control of PDE systems. In boundary feedback control, the actuation is applied only at the boundary of the spatial domain. This makes implementation very simple in practical applications. In this context, backstepping methods have been widely explored in the literature for PDE control \cite{boskovic2001boundary},\cite{1470714}. In \cite{Li2017FTStability}, the problem of finite-time (FT) stability and stabilization for distributed parameter systems is addressed.  In \cite{9980532}, a reinforcement learning-based boundary control strategy that uses a solution to a spatial Riccati-like equation for parabolic PDEs is presented. In \cite{9134871}, a saturated feedback control has been derived, which locally stabilizes a linear reaction-diffusion equation. Distributed actuator selection by proposing a primal–dual algorithm that achieves optimality has been addressed in \cite{8392433}. 

System design parameters influence the system performance through the system dynamics \cite{skelton1989model}. Traditional system design (optimize system parameters and then control) approach results in a sub-optimal system \cite{fathy2001coupling}. Optimizing system design and control gains simultaneously may lead to an optimal system \cite{10.1115/1.4027335}. This approach is known as the Control Co-design (CCD) process. While CCD has been extensively studied for lumped-parameter systems (i.e., ODE-based models)\cite{10.1115/1.4027335,8812969,chanekar2018co}, its extension to PDE systems is largely unexplored. The existing literature regarding CCD is mainly concentrated on actuator/sensor location as the design variable  \cite{7946139},\cite{15M1014759}. \cite{18M1171229}.  This research gap motivates the present work.
 While substantial research exists on PDE stabilization and optimal control, most existing approaches assume fixed plant parameters and focus exclusively on controller synthesis. In contrast, this work considers the simultaneous optimization of physical design parameters and controller gains under the control co-design framework. 

\textit{ Contributions:}
The main contributions of this paper are:

\begin{enumerate}
\item We formulate a control co-design (CCD) problem for a class of one-dimensional parabolic PDE systems with Neumann boundary feedback.  We derive sufficient stability conditions for the PDE system in terms of the system's design parameters and controller gains. 

\item Using the derived sufficient stability conditions and discretization of the PDE in the spatial domain we propose a tractable approximate CCD formulation. We then develop a gradient-based optimization algorithm for computing optimal design and stabilizing controller parameters for the PDE system.

\item We show the efficacy of our proposed formulation and solution procedure through a numerical example showcasing different types of CCD problems.
\end{enumerate}
The paper is organized as follows: Section~II presents the preliminaries and problem formulation. Section~III presents the CCD problem reformulation followed by solution procedure and numerical examples in Section~IV and Section~V. Section~VI concludes the paper.

\section{Preliminaries and Problem Statement}\label{prelim}
Consider\footnote{Notations: $\mathbb{R}$ denotes the set of real numbers. A matrix $A \in \mathbb{R}^{m \times n}$ has $m$ rows and $n$ columns. $A \in \mathbb{R}^{m}$ represents m - dimensional vector. The symbols $A^\top$ and $\operatorname{tr}(A)$ represent the transpose and trace of the matrix $A$, respectively. $||.||$  represents the Euclidean norm. For a symmetric matrix $P \in \mathbb{R}^{n \times n}$, $P \succeq 0$ (resp. $P \succ 0$) indicates that $P$ is positive semi-definite (resp. positive definite). $x_{\theta}$ denotes the partial derivative of x with respect to  $\theta$ i.e. $\frac{\partial x}{\partial \theta}$. The space \( L^2(0,1) \) is a Lebesgue space of square-integrable functions over the interval \( (0,1) \) defined as: $ L^2(0,1):= \left\{ f : (0,1) \to \mathbb{R} \;\middle|\; \int_0^1 |f(\xi)|^2 \, d\xi < \infty \right\}$. The Sobolev space is defined as $H^1(0,1) := \left\{ x \in L^2(0,1) \;\middle|\; x_\xi \in L^2(0,1) \right\}$. $C([0,\infty); H^1(0,1))$: space of continuous functions from $[0,\infty)$ into $H^1(0,1)$. For notational simplicity, the PDE state $x(\xi,t)$ and its spatial derivative $x_\xi(\xi,t)$ are denoted by $x$ and $x_\xi$, respectively. Furthermore, for a fixed spatial point $\bar{\xi}\in[0,1]$, the quantities $x(\bar{\xi},t)$ and $x_\xi(\bar{\xi},t)$ are denoted by $x(\bar{\xi})$ and $x_\xi(\bar{\xi})$, respectively, whenever the time dependence is clear from the context. For a function $f$, notation $\grad_{\theta}f$ denotes the gradient of $f$ with respect to $\theta$. J denotes the cost function. $\operatorname{diag}(\begin{bmatrix}a_1, a_2, \ldots, a_n\end{bmatrix})$ denotes the diagonal matrix with entries $a_1, a_2, \ldots, a_n$ on its main diagonal. Let $I$ denote the identity matrix of appropriate dimension.} a  system described by the parabolic PDE \cite{strauss2007partial} in one spatial dimension (1D),
\begin{align} \label{dynamics1}
    &x_{t}(\xi,t) = ax_{\xi\xi}(\xi,t) + b x(\xi,t),\;  x(\xi,0) = x_0(\xi) \nonumber\\
    &x_{\xi}(0,t) = u_1  , \;x_{\xi}(1,t) = u_2     \\
    &\text{where}\; \;  u_1 = k_1x(0,t),\; u_2 = k_2x(1,t).\nonumber
    \end{align}
where $t>0$ denotes time and $\xi\in[0,1]$ denotes the spatial variable, $ a\in \real,\; b \in \real$ are the system parameters. 
 $x(\xi,t) \in C\big([0,\infty); H^1(0,1)\big)$ is the system state variable with $x_0(\xi)$ as the initial condition. The control $u = [u_1\; u_2]\T \in \real^2$ is applied to the boundary of the system with $u_1$ applied at $\xi =0$ and  $u_2$  at $\xi =1$. The parameters $k_1\in\real,\; k_2 \in \real$ are the stabilizing linear static feedback control gains. The system has the equilibrium state $x_e(\xi)\in \real.$ Note that $x_e = 0$ is one of the equilibrium states of the dynamics \eqref{dynamics1}.

 The CCD procedure involves simultaneous optimization of the system parameters and the control gains \cite{10.1115/1.4027335}. Mathematically, we represent the  CCD optimization problem as,
\begin{align}
 \label{ccd1}
&\notag \min_{d,\,k_1,\,k_2}  \; J = f_d(d)+ \int_{0}^{\infty} \int_{0}^{1}
\left(qx^2 + ru\T u\right)\td\xi \td t\\
 \notag \text{s.t.} \quad &x_t = ax_{\xi\xi} +bx,\; x(\xi,0) = x_0(\xi),\\
& x_\xi(0,t) = k_1x(0,t)=u_1,\; x_\xi(1,t) = k_2x(1,t)=u_2,   \\
&\notag x(\xi,t) \to x_e(\xi) \; \text{ as }\; t \to \infty,\\
&\notag (\underline{d},\,\underline{k}_1,\,\underline{k}_2)  \leq (d,\,k_1,\,k_2)\leq (\overline{d},\,\overline{k}_1,\,\overline{k}_2),\\
&\notag (\xi,\;t) \in [0,\;1]\times [0,\; \infty), \; u = [u_1 \; u_2]\T.
\end{align}
Here, $d\in\real^{n_d}$ is the design parameter and $n_d$ is the number of design parameters. $ f_d : \real^{n_d} \rightarrow \mathbb{R} $ is the design objective function. 
$q \in \real,\; q \geq 0$, \quad $r \in \real,\; r>0$
are constant weighting coefficients.
The term $\int_{0}^{\infty} \int_{0}^{1}
\left(qx^2 + ru\T u\right)\td\xi \td t$ quantifies the quadratic control cost of the system \eqref{dynamics1}.
The constraint $x(\xi,t) \to x_e(\xi) \; \text{ as }\; t \to \infty$ ensures the stability of the co-designed system. The problem \eqref{ccd1} is a time-dependent, nonconvex and nonlinear optimization problem that is generally computationally challenging to solve. For system  \eqref{dynamics1}, $0\leq n_d\leq 2$ and the design variable $d$ can be $a$ or $b$ or $[a \quad b]\T.$. For a fixed $d$, we have $n_d=0$ and  \eqref{ccd1} becomes an optimal control gain synthesis problem. When $f_d(d) =0$, the design parameter $d$ influences the system CCD synthesis process through the system dynamics \eqref{dynamics1}.  Next, we present the reformulation of problem \eqref{ccd1}.


\section{CCD Problem Reformulation}
\label{prob-reformulation}

The CCD problem formulated in \eqref{ccd1} is an infinite-dimensional optimization problem constrained by a parabolic PDE. A direct solution to \eqref{ccd1} is a computationally challenging task.  In this section, we derive a computationally tractable approach to solve the CCD problem. We first establish sufficient stability conditions for the closed-loop PDE in terms of the system parameters and boundary feedback gains. We then discretize the PDE in the spatial domain to get a system of linear ordinary differential equations (ODEs). Finally using standard results from control theory we obtain a time-independent optimization problem with algebraic constraints.


\subsection{PDE System Stability}
\label{pde-stability}
We use the Lyapunov functional approach for infinite-dimensional systems \cite{CurtainZwart,Pazy} to derive the (sufficient) stability conditions for system \eqref{dynamics1} in terms of $a,\:b, \: k_1,\:k_2$.
%
\begin{theorem}\longthmtitle{Sufficient stability conditions for PDE \eqref{dynamics1}}\label{stability1}
Consider the PDE \eqref{dynamics1} with the equilibrium state  $x_e=0$.
Then the system is asymptotically stable about  $x_e=0$ if \\
$\kb_1 - a < 0, \quad 5\kb_1 -a +4b < 0,\quad
2a k_2+\kb_1 + a <0 $\\
where $\kb_1 = \max\{0, \; -ak_1\}$.
\end{theorem}
\begin{proof}
Consider the Lyapunov functional and its time-derivative
$$V = \frac{1}{2} \int_{0}^{1} x^2(\xi, t) \, \td\xi, \quad \dot{V} = \int_{0}^{1} x x_{t}\td\xi $$
Now, using \eqref{dynamics1}, it follows that
\begin{align*}
\dot{V} & = a \int_0^1 x x_{\xi\xi} \, \td\xi + b \int_0^1 x^2 \, \td\xi \\
        &= a \left[ x(1)x_\xi(1) - x(0)x_\xi(0) \right] - a \int_0^1 x_\xi^2 \, \td\xi + b \int_0^1 x^2 \, \td\xi.
        \end{align*}
Using boundary conditions in \eqref{dynamics1},
\begin{align*}
 &x_\xi(0) = k_1 x(0) \; \text{and} \;x_\xi(1) = k_2 x(1), \\
&\dot{V} 
= a k_2 x(1)^2 - a k_1 x(0)^2 - a \int_0^1 x_\xi^2 \, \td\xi + b \int_0^1 x^2 \, \td\xi \\
&\text{Using} \; \kb_1 = \max\{0, -ak_1\}, \text{we obtain}\\
&\dot{V} \leq a k_2 x(1)^2 +  \kb_1 x(0)^2 - a \int_0^1 x_\xi^2 \, \td\xi + b \int_0^1 x^2 \, \td\xi \\
&\text{Using Agmon's inequality:} \\ &x(0)^2 \leq x(1)^2 + 2 \sqrt{\int_0^1 x_\xi^2\,\td\xi} \sqrt{\int_0^1 x^2\,\td\xi}, \; \text{it follows that}\\
&\dot{V} \leq a k_2 x(1)^2 + \kb_1 \left[ x(1)^2 + 2 \sqrt{\int_0^1 x_\xi^2\,\td\xi} \sqrt{\int_0^1 x^2\,\td \xi} \right] \\
        &\quad - a \int_0^1 x_\xi^2 \, \td\xi + b \int_0^1 x^2 \, \td\xi \\
&\text{Using Young's inequality:}\; \text{if}\; a\geq 0,\; b\geq 0 \;\text{then}\\& 2ab \leq a^2 + b^2, \;\text{it follows that} \\
&\dot{V} \leq (a k_2 +  \kb_1) x(1)^2 +  \kb_1 \left[ \int_0^1 x_\xi^2 \, \td\xi + \int_0^1 x^2 \, \td\xi \right] \\
        &\quad - a \int_0^1 x_\xi^2 \, \td\xi + b \int_0^1 x^2 \, \td\xi \\
        &= (a k_2 +  \kb_1) x(1)^2 -  (-\kb_1 +a) \int_0^1 x_\xi^2 \, \td\xi \\
        &+ ( \kb_1 + b) \int_0^1 x^2 \, \td\xi \\
&\text{If}\;\;  \kb_1 - a < 0\;\; \text{and using Poincaré inequality: }\\&\text{for any x continuously differentiable on [0,\;1]} \\ &\int_0^1 x^2\,d\xi \leq 2x(1)^2 + 4\int_0^1 x_\xi^2\,\td\xi, \; \text{it follows that} \\
&\dot{V} \leq (a k_2 + \kb_1) x(1)^2 +  (\kb_1 + b ) \int_0^1 x^2 \, \td\xi \\
        &\quad + (a - \kb_1 ) \left[ \frac{x(1)^2}{2} -  \int_0^1 \frac{x^2}{4} \, \td\xi \right] \\
        &= \left(a k_2 + \frac{\kb_1}{2} + \frac{a}{2}\right) x(1)^2 + \left(\frac{5}{4}\kb_1 - \frac{a}{4} + b  \right) \int_0^1 x^2 \, \td\xi 
\end{align*}
If $\kb_1 -a < 0, \;2a k_2 + \kb_1 + a < 0, \; 5  \kb_1 - a + 4 b < 0$
then $\dot{V} < 0$, which implies asymptotic stability in $L^2$- norm of the equilibrium state $x_e=0$ in the Lyapunov sense for the parabolic PDE system.
\end{proof}

The stability condition for the homogeneous version of the PDE \eqref{dynamics1} (i.e., with $b=0$ in \eqref{dynamics1}), is stated next.
\begin{corollary}\longthmtitle{Stability of homogeneous version of  \eqref{dynamics1}}\label{stability1-homogeneous}
 The system  \eqref{dynamics1}  with $b=0$ is asymptotically stable about the equilibrium $x_e=0$ if 
 $$  \kb_1 - a < 0,\quad  5\kb_1 - a < 0,\quad
 2ak_2 + \kb_1 + a < 0$$
where $\kb_1 = \max\{0, -ak_1\}$.
\end{corollary}
Note that Theorem \ref{stability1} and Corollary \ref{stability1-homogeneous} provide a set of sufficient conditions only which form a  feasible set for the CCD problem \eqref{ccd1}. To compute an optimal solution, we first reformulate problem \eqref{ccd1}  into a computationally tractable format in the next sub-section.
\subsection{Reformulated CCD problem}
\label{reform1}
We first semi-discretize the PDE in \eqref{ccd1} in the spatial domain to obtain a set of ordinary differential equations (ODEs). We also convert the double integral objective function in \eqref{ccd1} into the standard single integral quadratic optimal control objective function \cite{lewis2012optimal} using trapezoidal rule \cite{atkinson1978introduction}.
\begin{proposition}\longthmtitle{CCD problem discretization}\label{discrete1}
Consider the dynamics \eqref{dynamics1} and the problem \eqref{ccd1}. Let $\Delxi=\frac{1}{N-1}$ for some positive integer $N>1$. Then the system \eqref{dynamics1} can be rewritten as
$$\dot{X}(t) = A X(t), \quad X(0) = X_0,$$
for some $X(t) \in \mathbb{R}^N $, and $ A \in \mathbb{R}^{N \times N} $ with 
$$A = \scalebox{0.9}{$\frac{1}{\Delxi^2}$}\scalebox{0.65}{$
\begin{pmatrix}
-a -a k_1\Delxi + b\Delxi^2 & a & 0 & \cdots & 0 \\
a & -2a + b\Delxi^2 & a & \cdots & 0 \\
\vdots & \ddots & \ddots & \ddots & \vdots \\
0 & \cdots & a & -2a + b\Delta\xi^2 & a \\
0 & \cdots & 0 & a & -a + ak_2\Delta\xi + b\Delxi^2
\end{pmatrix}.$}$$ The cost $J$ in \eqref{ccd1} is rewritten as
$$J_d = f_d(d) + \int_0^\infty \left( X^\top Q_d X + U^\top R_d U \right) \td t$$ with some $U=KX,$
\begin{align*}
Q_d &= q\,\frac{\Delxi}{2} \,\operatorname{diag}\left(\begin{bmatrix}
   \frac{1}{2}& 1 \ldots 1& \frac{1}{2} 
\end{bmatrix}\right),\\
R_d &= r \,\operatorname{diag}\left(\begin{bmatrix}
   1&1
\end{bmatrix}\right),\;K  = \begin{pmatrix}
k_1 & 0 & \ldots & 0 \\
0 & \ldots &  0 & k_2
\end{pmatrix}.
\end{align*}
\end{proposition}
\begin{proof}
As $\xi \in [0, 1],$ partition [0, 1] in N-1 intervals $[\xi_{i-1}, \xi_i], \;i = 1, \dots,N$ of equal lengths with $\Delta \xi = \frac{1}{N-1}.$
Discretizing system~\eqref{dynamics1} in spatial domain\text{\cite{Thomas1995}}, where $x_i(t)$ represents $x(\xi_i, t)$ we have 
\begin{align}  \label{discretized system}  
 \dot{x_i}(t)  = a\frac{x_{i+1}(t)-2x_{i}(t)+x_{i-1}(t)}{(\Delta\xi)^2} + bx_i(t) \notag\\
 = a\frac{x_{i+1}(t)}{(\Delta\xi)^2}+ a\frac{x_{i-1}(t)}{(\Delta\xi)^2}
 +\left(b-\frac{2a}{(\Delta\xi)^2}\right)x_{i}(t)\\
 i = 1,2,....,N \notag
 \end{align}
 Now, to apply the Neumann boundary conditions \eqref{dynamics1} using finite differences, we introduce \emph{ghost nodes} at \(i = 0\) and \(i = N+1\), located just outside the physical domain. These fictitious nodes are used to approximate the spatial derivatives at the boundaries and are eliminated using the boundary conditions.
 Using backward and forward difference \cite{Thomas1995} at $\xi = 0 ,\; \xi = 1 $ respectively, it follows that
 \begin{align} \label{discretitized boundary condition}
\frac{x_1(t) - x_0(t)}{\Delta\xi} = k_1x_1(t),\;\;
 \frac{x_{N+1}(t) - x_N(t)}{\Delta\xi} = k_2x_N(t)
  \end{align}
 using \eqref{discretized system} and  \eqref{discretitized boundary condition} , we have
\begin{align} \label{system of odes}
\dot{x}_1(t) &=  a \dfrac{x_2(t)}{(\Delta\xi)^2} + \left( b + \dfrac{-a - a k_1 \Delta\xi}{(\Delta\xi)^2} \right) x_1(t),\notag\\ 
\dot{x}_i(t) &= a \dfrac{x_{i+1}(t) + x_{i-1}(t)}{(\Delta\xi)^2} + \left( b - \dfrac{2a}{(\Delta\xi)^2} \right) x_i(t);\\
& i = 2, \ldots, N-1,\notag\\ 
\dot{x}_N(t) &= a \dfrac{x_{N-1}(t)}{(\Delta\xi)^2} + \left( b + \dfrac{-a + a k_2 \Delta\xi}{(\Delta\xi)^2} \right) x_N(t) \notag, 
\end{align}
For $X = [ x_1 \; x_2 \ldots x_{N-1} \;x_N]\T \in \mathbb{R}^N$, \eqref{system of odes} becomes the system of ODEs as \\
$$\dot{X} = AX$$ 
where
$$
A = \frac{1}{\Delta\xi^2}
\scalebox{0.65}{$
\begin{pmatrix}
-a -ak_1\Delta\xi + b\Delta\xi^2 & a & 0 & 0 \\
a & -2a + b\Delta\xi^2 & a  & 0 \\
\vdots & \ddots & \ddots  & \vdots \\
0 & \cdots & a & a \\
0 & \cdots & 0 &  -a + ak_2\Delta\xi + b\Delta\xi^2
\end{pmatrix}
$}
$$
The cost function is 
$  J = f_d(d) + \int_{0}^{\infty} \int_{0}^{1}(qx^2+ru\T u) \td\xi \td t = f_d(d) + \int_{0}^{\infty} \int_{0}^{1}\left(qx^2+r(u^2_1 + u^2_2)\right) \td\xi \td t.$ Using trapezoidal rule \cite{atkinson1978introduction} and  $u_1 = k_1x(0), \; u_2 = k_2x(1),$  from \eqref{dynamics1} we have
\begin{align*}
 J_d = f_d(d) +& \int_{0}^{\infty}\left(\frac{\Delta\xi}{2}\left[ qx_1^2 + 2\sum_{i = 2}^{N-1}qx_i^2 + qx_N^2\right]\td t\right) \\ 
&+ (\int_{0}^{\infty}\left([rk_1^2x_1^2 + rk_2^2x_N^2]\td t \right)) 
\end{align*}
Now, for $U = [u_1 \;\; u_2]\T = [k1x_1 \;\; k_2x_N]\T = KX$
$$ J_d =  f_d(d) +  \int_{0}^{\infty}\left( X\T (t)Q_dX(t) +U\T (t)R_dU(t)\right)\td t$$  $$= f_d(d) + \int_0^\infty X\T (Q_d+K\T R_dK)X \td t $$ 
where 
$$Q_d = q\frac{\Delta\xi}{2}  \operatorname{diag}\left(\begin{bmatrix}
   \frac{1}{2}& 1  \ldots  1& \frac{1}{2} 
\end{bmatrix}\right), \quad R_d = r \operatorname{diag}\left(\begin{bmatrix}
   1&1
\end{bmatrix}\right)$$

\end{proof}
Note that in Proposition \ref{discrete1}, $q\geq 0,\; r>0$ implies  $Q_d\succeq 0$ and $R\succ 0$. The system matrix $A$ depends on both the design parameter $d$ and the boundary feedback gains $k_1$ and $k_2$. One may write the discrete dynamic system in the standard feedback closed-loop format as follows.
\begin{remark}\longthmtitle{Relation to standard closed-loop system}\label{closed-loop}
		{\rm 
		In Proposition \ref{discrete1}, $A$ can be written as $A=A_0+BK$ where
$$ A_0 = \frac{1}{\Delta\xi^2}
\scalebox{0.65}{$
\begin{pmatrix}
-a  + b\Delta\xi^2 & a & 0 & \cdots & 0 \\
a & -2a + b\Delta\xi^2 & a & \cdots & 0 \\
\vdots & \ddots & \ddots & \ddots & \vdots \\
0 & \cdots & a & -2a + b\Delta\xi^2 & a \\
0 & \cdots & 0 & a & -a  + b\Delta\xi^2
\end{pmatrix}
$}, $$
$$B =  \frac{1}{\Delta\xi^2} \begin{pmatrix}
- a\Delta\xi & 0 \\
0 & 0 \\
\vdots & \vdots \\
0 &  0 \\
0 & a\Delta\xi
\end{pmatrix},\;
 K  = \begin{pmatrix}
k_1 & 0 & 0 & \cdots & 0 \\
0 & \cdots & 0 & 0 & k_2
\end{pmatrix}.
$$}		\oprocend
\end{remark} 
  Using Proposition \ref{discrete1} and Theorem \ref{stability1}, we write the discrete (in spatial coordinate)  version of the CCD problem \eqref{ccd1} as
\begin{align}
 \label{ccd2}
&\notag \min_{d,\,k_1,\,k_2}  \; J_d = f_d(d) + \int_0^\infty \left( X\T Q_d X + U\T R_d U \right) \td t\\
 \notag \text{s.t.} \quad &\dot{X} = A(d,k_1,k_2) X, \quad X(0) = X_0, \quad U=K(k_1,k_2)X,\\
& \kb_1 - a < 0, \quad 5\kb_1 -a +4b < 0,\\
&\notag 2a k_2+\kb_1 +a <0, \quad \kb_1 = \max\{0, \; -a k_1\},\\
& \notag(\underline{d},\,\underline{k}_1,\,\underline{k}_2)  \leq (d,\,k_1,\,k_2)\leq (\overline{d},\,\overline{k}_1,\,\overline{k}_2),\\
&\notag A(d,\,k_1,\,k_2)\text{  is Hurwitz}. 
\end{align}
The abstract condition $A$ to be Hurwitz ensures stability of the designed system, i.e., $\Vert X(t)\Vert \rightarrow 0$ as $t\rightarrow \infty$ \cite{khalil2002nonlinear}.   The problem  \eqref{ccd2} is non-convex, nonlinear, time-dependent, and with the abstract stability constraint is still challenging to solve. From optimal control theory \cite{lewis2012optimal} we have
$$\int_0^\infty \left( X^\top Q_d X + U^\top R_d U \right)\td t=\operatorname{tr}(P X_0 X_0^\top).$$
where $X_0 \in \mathbb{R}^{N}$ denotes the discretized initial condition vector obtained from the PDE initial profile $x_0(\xi)$ through the spatial discretization,  $P\succ 0$ is the solution to
\begin{align}
    \label{lyap1}
    A\T P + P\, A + Q_d + K\T R_d K = 0.
\end{align}
As $Q_d\succeq 0$ and $R_d\succ 0$, the Lyapunov equation \eqref{lyap1} admits a positive-definite solution $P$ when $A$ is Hurwitz
\cite{lewis2012optimal}.
Using \eqref{lyap1}, the problem \eqref{ccd2} is equivalently rewritten as
\begin{align}
 \label{ccd3}
&\notag \min_{d,\,k_1,\,k_2,\,P\succ 0}  \; J_f = f_d(d) + \operatorname{tr}(P X_0 X_0^\top), \\
 \notag \text{s.t.} \quad &A\T P + P\, A + Q_d + K\T R_d K = 0, P=P^\top\succ0.\\
& \kb_1 - a < 0, \quad 5\kb_1 -a +4b < 0,\\
&\notag 2a k_2+\kb_1 +a <0, \quad \kb_1 = \max\{0, \; -a k_1\},\\
&\notag (\underline{d},\,\underline{k}_1,\,\underline{k}_2)  \leq (d,\,k_1,\,k_2)\leq (\overline{d},\,\overline{k}_1,\,\overline{k}_2).
\end{align}
The problem \eqref{ccd3} is time-independent, nonconvex, and nonlinear with algebraic constraints. \eqref{ccd3}  is iteratively solved in a computationally tractable manner using a gradient-based method presented in the next section.
\section{Solution Procedure}
\label{sol-procedure1}
In this section we develop a gradient descent procedure \cite{Bertsekas}  to compute an optimal solution to the problem \eqref{ccd3}. We first define the following set,
\begin{align}
      \label{set-def1}
         \Dc=\Set{(d,\,k_1,\,k_2)|
         \begin{array}{ll}
(\underline{d},\,\underline{k}_1,\,\underline{k}_2)  \leq (d,\,k_1,\,k_2)\leq (\overline{d},\,\overline{k}_1,\,\overline{k}_2),\\
     \kb_1 - a < 0, \; 5\kb_1 -a +4b < 0,\\
2a k_2+\kb_1 +a <0,\\
\kb_1 = \max\{0, \; -ak_1\},\\
A(d,\,k_1,\,k_2)\text{  is Hurwitz}.
\end{array}}.
  \end{align}
  The set $\Dc$ represents the admissible set of stabilizing design and controller parameters and is the feasible set of \eqref{ccd3}.
  Using \eqref{set-def1}, the problem \eqref{ccd3} is written as
  \begin{align}
 \label{ccd4}
& \min_{(d,\,k_1,\,k_2)\in \Dc}  \; J_f = f_d(d) + \operatorname{tr}(P X_0 X_0^\top).
\end{align}
Note that \eqref{ccd4} is the `\textit{approximate}' version of the CCD problem \eqref{ccd1}. 
We develop a constrained gradient-based method to solve \eqref{ccd4}. We first compute the gradient of $J_f$ with respect to $d,\;k_1,\;k_2$ in the next result.
\begin{lemma}\longthmtitle{Gradient computation}\label{gradient1} Consider the system \eqref{dynamics1}, CCD problem \eqref{ccd1} and its final approximation \eqref{ccd4} with $d\in\real^{n_d}$. Then the gradient $\nabla\,J_f(d,k_1,k_2) = \Big[\frac{\pd J_f}{\pd d}\;  \frac{\pd J_f}{\pd k_1}\; \frac{\pd J_f}{\pd k_2}\Big]\T$ where $\frac{\pd J_f}{\pd d_j}=\frac{\pd f_d}{\pd d_j}+\tr\big(\frac{\pd P}{\pd d_j}X_0 X_0^\top\big)$ for $j=1,\dots,n_d$ and $\frac{\pd J_f}{\pd k_i}=\tr\big(\frac{\pd P}{\pd k_i}X_0 X_0^\top\big)$ for $i=1,2$ computed at some $(d^0,\,k_1^0,\,k_2^0)\in\Dc$. Where  $\frac{\pd P}{\pd d_j}$ is the solution to $$A\T \frac{\pd P}{\pd d_j}+\frac{\pd P}{\pd d_j}A+\frac{\pd A\T}{\pd d_j}P+P\frac{\pd A}{\pd d_j}=0,$$
and $\frac{\pd P}{\pd k_i}$ is the solution to $$A\T \frac{\pd P}{\pd k_i}+\frac{\pd P}{\pd k_i}A+\frac{\pd A\T}{\pd k_i}P+P\frac{\pd A}{\pd k_i}+\frac{\pd K\T}{\pd k_i}R_dK+KR_d\frac{\pd K}{\pd k_i}=0.$$
\end{lemma}
\begin{proof}
Differentiating $J_f$ and \eqref{lyap1} partially with respect to $d_j$ and $k_i$ gives the required result. 
\end{proof}
We now outline the constrained gradient descent procedure to solve \eqref{ccd4}.
\begin{algorithm}[H]
\caption{CCD Algorithm} 
\label{alg:GradientDescent}
\begin{algorithmic}[1]
\Require  $(d^0,\, k_1^0,\, k_2^0)\in \Dc$,  $Q_d$, $R_d$, $X_0$,  $\epsilon$, $\epsilon_1$, $\sigma$, $\beta$
\Ensure  $d^*$, $k_1^*$, $k_2^*$
\State Set: $j \gets 0$,  form $A(d^0,\, k_1^0,\, k_2^0)$
\State Compute $P^0$ using \eqref{lyap1} and $J^0_f$ using \eqref{ccd4}
\State Compute  $\nabla_k J_f^0=\nabla \,J_f(d^0,k_1^0,k_2^0)$ using Lemma \ref{gradient1}
\While{true} 
\State Compute step-size $s^j$ using Algorithm \ref{alg:Armijo}
\State $[{d\T}^{j+1} \; k_1^{j+1} \; k_2^{j+1}]\T=[{d\T}^{j} \; k_1^{j} \; k_2^{j}]\T\;-\;s^j\,\nabla\,J_f^j$ 
\State Compute $P^{j+1}$ using \eqref{lyap1} and $J^{j+1}_f$ using \eqref{ccd4}
\State Compute  $\nabla\, J_f^{j+1}=\nabla \,J_f(d^{j+1},k_1^{j+1},k_2^{j+1})$ using Lemma \ref{gradient1}
\If{ $\|\nabla_k J_f^{j+1}\| \leq \epsilon$ or $\vert J_f^j-J_f^{j-1}\vert \leq \epsilon_1$ }
    \State Go to Step 17
    \Else 
    \State $(d^j,\,k_1^j,\,k_2^j)\gets (d^{j+1},\,k_1^{j+1},\,k_2^{j+1})$,  $J_f^j\gets J_f^{j+1}$
    \State $\nabla\,J_f^j\gets \nabla\, J_f^{j+1}$, $j\gets j+1$
    \State Compute  $\nabla_k J_f^j$ using Lemma \ref{gradient1}
\EndIf
\EndWhile
\State \Return  $d^* \gets d^{j+1}$, $k_1^* \gets k_1^{j+1}$, $k_2^* \gets k_2^{j+1}$
\end{algorithmic}
\end{algorithm}
Next, we state the procedure for computing the step-size $s^j$ used in Step 6 of the Algorithm \ref{alg:GradientDescent}. The procedure is based on the Armijo step-size selection rule \cite{Bertsekas}.
\begin{algorithm}[H]
	\caption{Computation of step-size $s^j$}\label{alg:Armijo}
	\begin{algorithmic}[1]
		\Require $\beta,\; \sigma,\; d^j,\;k_1^j,\;,k_2^j,\;\nabla\, J_f^j,\; J_f^j$
		\Ensure $s^j$
        \State Set $s \gets 1$
        \While{1}
        \State $[{d\T}^{j+1} \; k_1^{j+1} \; k_2^{j+1}]\T=[{d\T}^{j} \; k_1^{j} \; k_2^{j}]\T\;-\;s\,\,\nabla\,J_f^j$ 
                 \If{$(d^{j+1},k_1^{j+1},k_2^{j+1})\in \Dc$}
\State Compute $P^{j+1}$ using \eqref{lyap1} and $J^{j+1}_f$ using \eqref{ccd4}
          \If{$J_f^{j+1} \leq J_f^j - \sigma s\Vert\nabla\, J_f^j\Vert^2$}
                \State Go to Step 15
            \Else
                \State $s \gets \beta s$
            \EndIf
        \Else
             \State $s \gets \beta s$
        \EndIf            
\EndWhile
\State \Return $s^j\gets s$
\end{algorithmic}	
\end{algorithm}
Typically  $\sigma \in [0.1,\,0.5],\;\beta \in [0.2,\,0.5]$  in Algorithm \ref{alg:Armijo}. Note that the  optimal solution   $(d^*,\, k_1^*,\, k_2^*)$
  computed from Algorithm \ref{alg:GradientDescent} for the CCD problem \eqref{ccd4} lies in the  set $\Dc$. By definition of $\Dc$, the sufficient stability conditions of Theorem \ref{stability1} are satisfied. Therefore, $(d^*,\,k_1^*,\,k_2^*)$ is a  stabilizing parameter set for the PDE system \eqref{dynamics1} and a feasible solution for \eqref{ccd1}.
\begin{remark}\longthmtitle{Relation to original CCD problem \eqref{ccd1}}\label{closed-loop-ccd}
		{\rm 
        The problem \eqref{ccd4} uses the discretized version of the PDE \eqref{dynamics1} and is an approximate version of the original CCD problem \eqref{ccd1}. The optimal solution to \eqref{ccd4} using Algorithm \ref{alg:GradientDescent} is a feasible solution to \eqref{ccd1}. Our numerical studies (Section \ref{examples}) show that this solution improves the performance of the PDE system \eqref{dynamics1} with respect to the CCD objective in \eqref{ccd4}.   
	A detailed analytical study of the CCD Algorithm \ref{alg:GradientDescent} with respect to the original CCD problem \eqref{ccd4} regarding its convergence and optimality properties is part of our future work.}		\oprocend
\end{remark} 
 Next, we present examples to justify our proposed theory.

\section{Examples}\label{examples}
The dynamics described by \eqref{dynamics1} represent real-world practical systems like heat transfer process,  the gas diffusion process, etc., \cite{strauss2007partial}. To demonstrate the efficacy of our proposed theory, we consider the following scenarios.
\begin{enumerate}[leftmargin=0.1cm, labelsep=0.1cm, align=left]
    \item[\textbf{Case 1}:] Given $a,\,b$ compute optimal $k_1, \,k_2$.
    \item[\textbf{Case 2}:] Given $b$, compute  $d=a$ and $k_1,k_2$  with $f_d(a)=0$.
   \item[\textbf{Case 3}:] Given $b$, compute  $d=a$ and $k_1,k_2$  with $f_d(a)=a^2$.
   \item[\textbf{Case 4}:] Compute $d = [a\;b]\T$ and  $k_1, k_2$ with  $f_d(d) = 0$.
   \item[\textbf{Case 5}:] Compute $d = [a\;b]\T$ and  $k_1, k_2$  with $f_d(d) = a^2 +b^2$.
\end{enumerate}

Case 1 is the standard controller design problem. 
In Cases 2 and 4 are CCD problems with no design objective with the design variables affecting the system performance through the dynamics. In Cases 3 and 5, we perform CCD with design objective $f_d$. 

We use Algorithm \ref{alg:GradientDescent}  for the CCD synthesis process.  The values of the relevant parameters used for the CCD computation process are taken as: 
\begin{align*}
    &q =10,\; r =10^2, \; \epsilon=10^{-3},\; \epsilon_1=10^{-6},
\sigma=0.5,\;\beta=0.5. 
\end{align*}
The spatial domain is discretized into $N=50$ equally spaced grid points, $ \xi_i=\frac{i-1}{N-1},\qquad i=1,2,\ldots, 50. $

The initial state is as follows
\begin{align*}
x_0(\xi) &= \sin(\pi\xi),\\
X_0 &= \begin{bmatrix}
x_0(\xi_1)&
x_0(\xi_2)&
\cdots &
x_0(\xi_N)
\end{bmatrix}^T\\
&=
\begin{bmatrix}
\sin(\pi\xi_1)&
\sin(\pi\xi_2)&
\cdots&
\sin(\pi\xi_N)
\end{bmatrix}^T.
\end{align*}

 The initial control gain values are taken as $k_1^0 = 3 ,\; k_2^0 = -3$. For Case 2 and Case 3, the initial design parameter is taken as $a^0 = 1 $  and known $b=0.2$. For Case 4 and Case 5, the initial design parameters are taken as $a^0 = 1,\; b^0 = 0.2$ and 
 $[\underline{a}\; \underline{b}]^T \leq [a\; b]^T\leq [\overline{a} \;\overline{b}]^T$ where  $\underline{a}  = 1, \; \underline{b} = -10, \;\;\overline{a} = 10,\;\overline{b} =50$. 
 Our numerical analysis procedure is as follows.
\begin{itemize}
    \item We compute the initial and optimal values of the cost function in \eqref{ccd4}, i.e., $J_f^0,\; J_f^*$ respectively.
    \item We compute the  optimal variable values $a^*,\; b^*,\;k_1^*,\; k_2^*$.
    \item We solve the PDE  numerically using the MATLAB function `\textit{pdepe}' \cite{MATLAB1} for the spatial domain $\xi \in [0,\;1]$ and the time domain $t\in [0, \;30]$. The spatial domain was discretized in $50$ parts and the time domain was discretized in $2000$ parts.
    \item We compute the initial and the final cost function values in \eqref{ccd1}, i.e., $J^0,\; J^*$. This is done by numerically computing the double integral using the MATLAB function `\textit{trapz}'. The final cost is computed for $a^*,\; b^*,\;k_1^*,\; k_2^*$.
   \item We also compute the initial and final values of the control performance objective
\[
J_u=\int_{0}^{\infty}\int_{0}^{1}
\left(qx^2+ru^\top u\right)\,d\xi\,dt,
\]
denoted by $J_u^0$ and $J_u^*$, respectively. Note that $J_u$ represents only the quadratic control performance term and excludes the design objective term $f_d(d)$.
\end{itemize}
All computations are performed using MATLAB \cite{MATLAB1}. We compile all the results in Table \ref{tab:cases_updated}. Next, we discuss our results.

\begin{table}[!t]
\centering
\renewcommand{\arraystretch}{1.5}
\resizebox{\columnwidth}{!}{%
\begin{tabular}{c c c c c c c c c c c}
\hline
Case & $a^*$& $b^*$ & $k_1^*$ & $k_2^*$ & $J_f^0$ & $J_f^*$ &
$J^0$ & $J^*$ & $J_u^0$ & $J_u^*$ \\
\hline
1 & --  & --  & 0.5 & $-0.5$ & 39.5 & 13 & 39.9   & 14.7  & 39.9   & 14.7   \\
2 & 6.5 & --& 0.5 & $-0.5$ & 39.5 & 1.6 & 39.9  & 1.8  & 39.9    & 1.8  \\
3 & 1.7& --  & 0.5&  $-0.5$ & 40.5 & 9.8 & 40.9 & 10.7 & 39.9 & 7.98 \\
4 &6.1& $-1.5$& 0.5&  $-0.5$& 39.5  &  1.3 & 39.9  &  1.4&  39.9& 1.4\\
5 &1.6 & 0.0 &2.9  &$-2.9$  & 40.6 & 25.2 & 41 & 25.2  & 39.9 & 22.6 \\
\hline
\end{tabular}}
\caption{Comparison of Initial and Optimized Performance Indices}\label{tab:cases_updated}
\end{table}

\begin{figure}[!t]
\subfloat[Uncontrolled]
{
    \includegraphics[width=0.15\textwidth,height=0.15\textwidth]{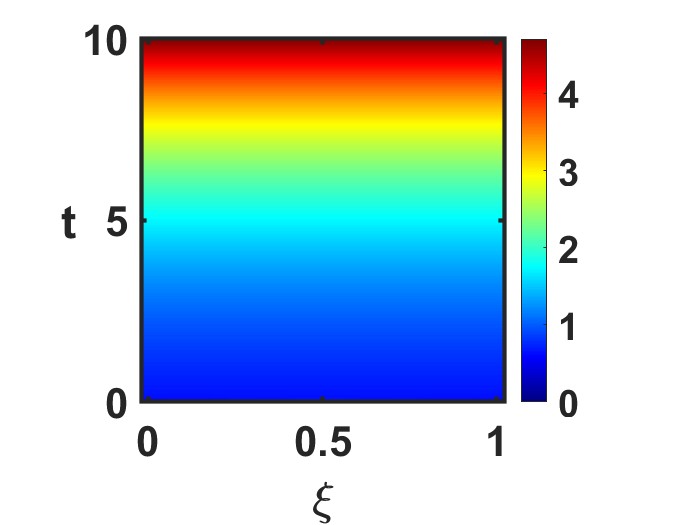}
    \label{fig:6.1}
}
\subfloat[CCD initialization]
{
    \includegraphics[width=0.15\textwidth,height=0.15\textwidth]{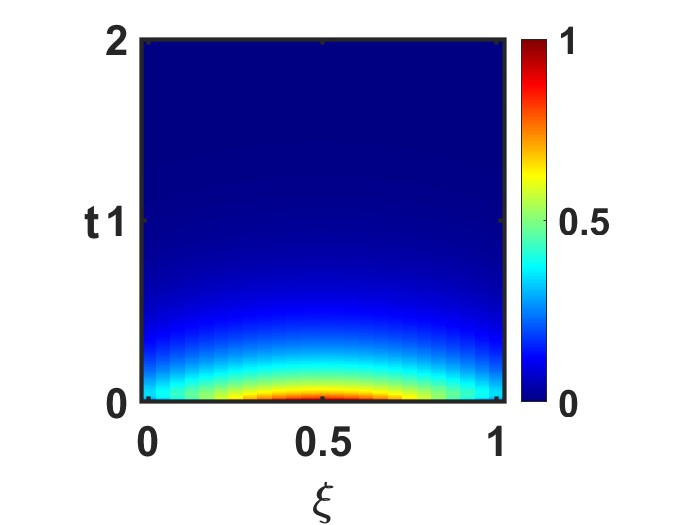}
    \label{fig:6.2}
}
\subfloat[Optimal CCD]
{
    \includegraphics[width=0.15\textwidth,height=0.15\textwidth]{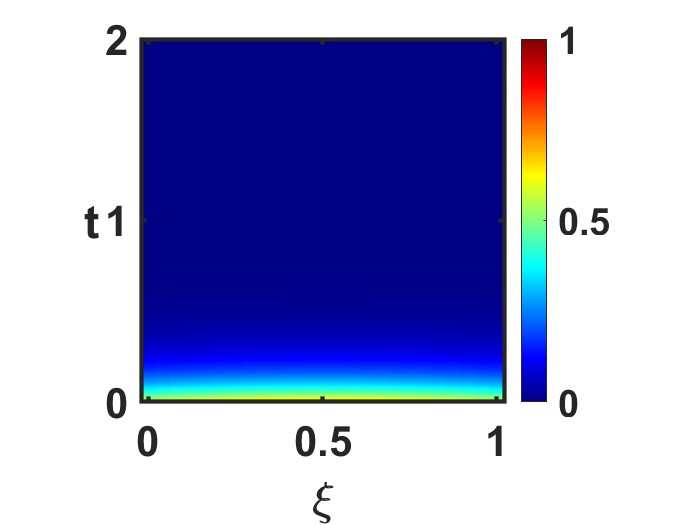}
    \label{fig:6.3}
}
\caption{State evolution for boundary feedback gains for Case 3.}
\label{fig:state_comparison}\
\end{figure}
\begin{figure}[!t]
 \subfloat[Case 1.]
{
    \includegraphics[width=0.15\textwidth,height=0.15\textwidth]{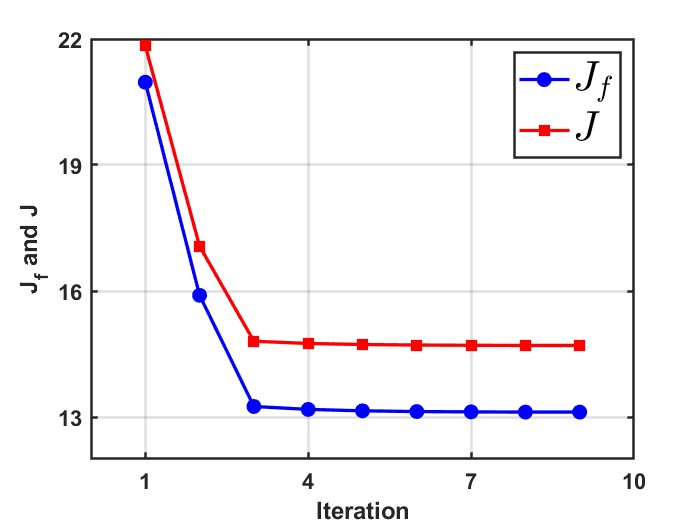}
    \label{fig:6.4}
}
\subfloat[Case 2.]
{
\includegraphics[width=0.15\textwidth,height=0.15\textwidth]{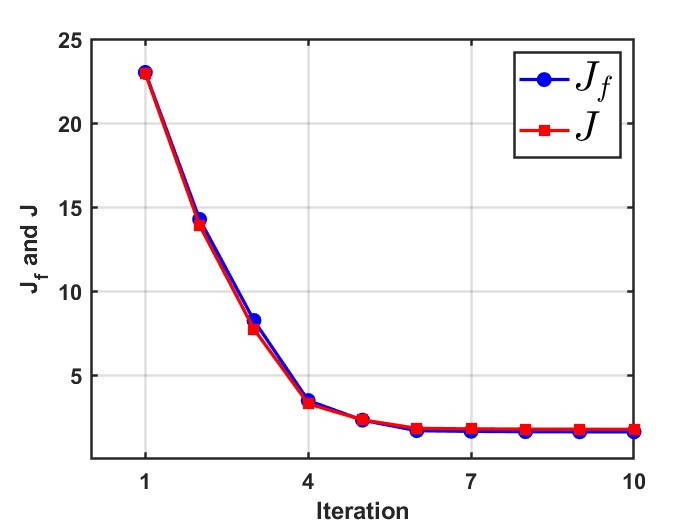}
    \label{fig:6.5}
}
\subfloat[Case 4.]
{
    \includegraphics[width=0.15\textwidth,height=0.15\textwidth]{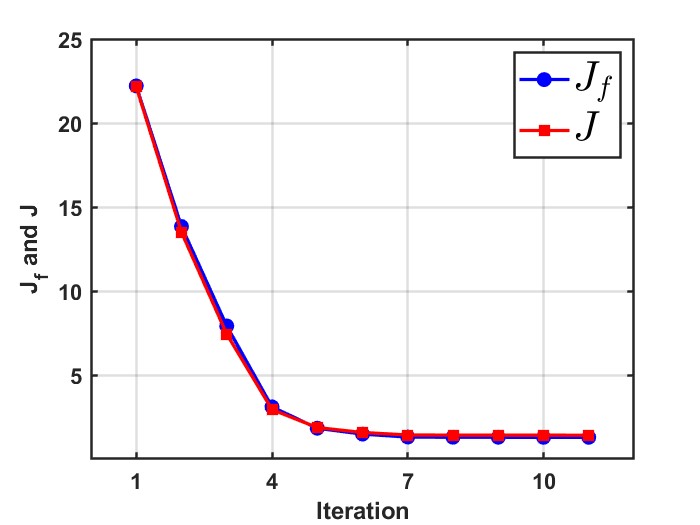}
    \label{fig:6.6}
}
\caption{Iteration-wise comparison of discrete and continuous cost functions as per CCD Algorithm \ref{alg:GradientDescent} for three cases.}
\label{fig:cost_plots}
\end{figure}
 \subsection{Discussion}
We discuss our results case by case. Fig. \ref{fig:6.1} shows the state evolution heat map of the unstable system, i.e., when $k_1=k_2=0$. Fig. \ref{fig:6.2} shows the state evolution heat map of the initial system for the CCD Algorithm \ref{alg:GradientDescent} which is the same for all the cases. Note that the initial system is stable but not optimal. 

\textbf{Case 1:} This case corresponds to the conventional controller design problem with fixed plant parameters. It serves as the baseline for comparison with the proposed CCD formulations.

\textbf{Case 2:} The diffusion coefficient $a$ is optimized together with the controller gains. Compared with Case~1, the additional design freedom significantly improves the system performance $J$. 

\textbf{Case 3:} The design objective $f_d(a)=a^2$ penalizes large diffusion coefficients. Here, the system performance $J$  improves after CCD; however, compared to Case 2, the control performance  $J_u$ is worse. This is due to the trade-off for the smaller value of $a$. 


\textbf{Case 4:} Simultaneous optimization of $a$, $b$ together with $k_1,\, k_2$  further improves  the system performance compared to Case 2. This is the best performing system with respect to $J_u$. The state evolution heat map for the optimal system is shown in Fig. \ref{fig:6.3}, illustrating the improved system performance.


\textbf{Case 5:} Introducing  $f_d(a,b)=a^2+b^2$ in the CCD process results in smaller values of $a$ and $b$. However, the trade-off is that this case has the worst control performance $J_u$. 

Fig. \ref{fig:cost_plots} shows the decrease in $J$ as $J_f$ decreases iteratively as per the CCD Algorithm \ref{alg:GradientDescent}. 
Note that the close agreement between $J_f$ and $J$ shows that the finite-dimensional optimization is a close approximation of the original PDE system.
The results demonstrate the efficacy and utility of our proposed  CCD formulation and solution procedure.
Depending on the practical design requirements, different CCD formulations can be adopted, as illustrated by the cases above.

\section{Conclusion}
In this paper, we have presented a control co-design (CCD) framework for a class of systems with parabolic PDE dynamics, for simultaneous optimization of plant and controller parameters.
 We have derived sufficient stability conditions for the PDE system and used them to develop a gradient-based CCD synthesis algorithm. The effectiveness of the proposed CCD framework is demonstrated through numerical examples. Our future work will be focused on the analytical study of the CCD algorithm with respect to convergence, optimality and scalability properties with extension to multi-dimensional spatial domains. 

\bibliographystyle{IEEEtran}
\bibliography{ref1}
\end{document}